\documentclass[12pt]{article}
\usepackage[centertags]{amsmath}
\usepackage{amsfonts}
\usepackage{amssymb}
\usepackage{amsthm}
\usepackage{amsmath,mathrsfs}
\usepackage{amsmath,amscd}
\usepackage[matrix,arrow,curve,cmtip]{xy}
\title{\textbf{Higher Galois for Segal Topos and Natural Phenomena}}
\author{Renaud Gauthier \footnote{rg.mathematics@gmail.com} \\ \\}
\theoremstyle{definition}
\newtheorem*{acknowledgments}{Acknowledgments}
\newtheorem{PiGpd}{Theorem}[section]
\newtheorem{istarwelldef}[PiGpd]{Lemma}
\newtheorem{istarff}[PiGpd]{Lemma}
\newtheorem{RuHomATSeT}[PiGpd]{Lemma}
\newtheorem{iTLadj}[PiGpd]{Theorem}

\DeclareMathOperator*{\adj}{\rlarr}

\newcommand{\beq}{\begin{equation}}
\newcommand{\eeq}{\end{equation}}

\newcommand{\rarr}{\rightarrow}

\newcommand{\rlarr}{\rightleftarrows}

\newcommand{\Rarr}{\Rightarrow}

\newcommand{\xrarr}{\xrightarrow}

\newcommand{\cC}{\mathcal{C}}

\newcommand{\cU}{\mathcal{U}}

\newcommand{\cX}{\mathcal{X}}
\newcommand{\cY}{\mathcal{Y}}

\newcommand{\bL}{\mathbb{L}}

\newcommand{\bR}{\mathbb{R}}

\newcommand{\Cat}{\text{Cat}}

\newcommand{\Hom}{\text{Hom}}
\newcommand{\Ho}{\text{Ho}\,}

\newcommand{\Loc}{\text{Loc}}

\newcommand{\op}{\text{op}}

\newcommand{\Set}{\text{Set}}

\newcommand{\Top}{\text{Top}}
\newcommand{\uHom}{\underline{\Hom}}

\newcommand{\CatD}{\Cat_{\Delta}}

\newcommand{\dkAff}{\text{d}k\text{-Aff}}

\newcommand{\dStk}{\text{dSt}(k)}

\newcommand{\fSets}{\text{fSets}}

\newcommand{\Gal}{\text{Gal}}

\newcommand{\hatXop}{\widehat{\cX^{\op}}}

\newcommand{\kMod}{k\text{-Mod}}

\newcommand{\Piinf}{\Pi_{\infty}}
\newcommand{\PiThat}{\Pi_{\infty, T}^{\wedge}}
\newcommand{\PiTX}{\Pi_{\infty,T}\hatXop}

\newcommand{\RHom}{\mathbb{R} \uHom}

\newcommand{\RHomgTX}{\mathbb{R} \uHom_{\SeT}(T, \cX)}
\newcommand{\RHomLgXT}{\mathbb{R}\uHom^*_{\SeT}(\cX, T)}
\newcommand{\RHomRgTX}{\mathbb{R} \uHom_{*, \SeT}(T, \cX)}
\newcommand{\RHomgeom}{\RHom^{\text{geom}}}
\newcommand{\RHomg}{\RHom_{\SeT}}
\newcommand{\RHomRg}{\RHom_{*,\SeT}}

\newcommand{\RuHom}{\bR \uHom}

\newcommand{\SetD}{\Set_{\Delta}}

\newcommand{\skCAlg}{\text{s}k\text{-CAlg}}

\newcommand{\SeT}{\text{SeT}}

\newcommand{\SeCat}{\text{SeCat}}
\newcommand{\SePC}{\text{SePC}}
\newcommand{\SeGpd}{\text{SeGpd}}

\newcommand{\tSetsG}{\text{tSets}^G}

\begin{document}
\maketitle
\begin{abstract}
In \cite{TV1}, Toen and Vezzosi show that $\RHomgeom(T,\cX)$ is a Segal groupoid, for $T$ a Segal topos, $\cX = \Loc(X)$ the Segal category of locally constant stacks on a CW complex $X$. Taking the realization of such a groupoid as in \cite{HS} defines a pro-object $H_T = |\RHomgeom(T, -)|$ that is defined to be the homotopy shape of the topos $T$ (\cite{TV1}). What we do instead is fix $\cX$, any Segal topos, and let $T$ vary. We show that $\RHom^*_{lex}(\cX,T) = \RHomgeom(T,\cX)$ is a Segal groupoid, from which it follows that we have a universal map from $\cX$ to the Segal category of $\cX$-local systems on $\RHomgeom(\cX,\cX)$, in the spirit of \cite{Hoy} where it is proved, morally, that local systems on $H_T$ are equivalent to $T$ itself. We provide one application of this formalism to the setting where $\cX=\dStk$ is the Segal topos of derived stacks, $k$ a commutative ring. We argue objects of $\dStk$ correspond to manifestations of natural laws, themselves modeled by simplicial algebras, objects of $\skCAlg$ (\cite{TV2}, \cite{TV}, \cite{TV4}, \cite{T}).
\end{abstract}

\newpage

\section{Introduction}
From the perspective of shape theory, one can define the shape of a Segal topos $T$, as Toen and Vezzosi did in \cite{TV1}, as being defined by $H_T = |\RHomgeom(T,-)|$. Here $|\quad |$ denotes the realization of a Segal category as in \cite{HS}, left adjoint of the fundamental groupoid functor $\Piinf$. In \cite{TV1} it is proven that for $X$ a CW complex, we have for $T = \Loc(*) = \Top$, $X \simeq |\RHomgeom(\Top, \Loc(X))|$, or equivalently $\Piinf(X) \simeq \RHomgeom(\Top, \Loc(X))$. In SGA 1 however, Grothendieck suggests that the category of fiber functors from $\cX$ to $T$ in $\SeT$, $\RHomLgXT = \RHomgeom(T,\cX)$, could be called a fundamental groupoid, thereby providing a generalization of the above result from \cite{TV1}. This can be made precise, as we will do in this work, but this won't be the fundamental $\infty$-groupoid of $\cX$ itself, rather it will be regarded as the fundamental $\infty$-groupoid of a Segal category of functors instead.\\

To see how this comes about, recall that for $T$ and $\cX$ two Segal topos, or more generally, two topos, a geometric morphism $f$ from $T$ to $\cX$ consists of a pair of adjoint functors, $f_*: T \rarr \cX$ and $f^*: \cX \rarr T$, $f^* \dashv f_*$, $f^*$ left exact. Thus if $\SeT$ denotes the category of Segal topos, we adopt the notation $\RHomgTX$ for $\RHomgeom(T, \cX)$, which one can define as $\RHomLgXT \simeq \RHomRgTX^{\op}$ depending on whether we fix our attention on left adjoints, or right adjoints. Now it turns out fiber functors are left exact, and working with Segal topos it is therefore natural, in the spirit of Grothendieck, to expect the groupoid of fiber functors in the Segal setting to be $\RHomLgXT$, sub-Segal category of $\RHom(\cX, T) = \hatXop$, relative to $T$, which is implied here. We prove, using Toen's work in \cite{TV1}, that $\RuHom^*_{\SeT}(\cX,T)$ is a Segal groupoid for any $T$. Hence we adopt the suggestive notation $\RHomLgXT = \PiTX$.\\

Considering local systems $i_T = \RHom(-,T)$ on it, we show we have a unique, universal map $\cX \rarr i_T \PiTX$. In particular if $T = \cX$, we have that for any Segal groupoid $\cY$, we have a morphism of Segal categories from $i_{\cX} \Pi_{\infty,\cX} \widehat{\cX^{\op}}$ to the internal perception $\RuHom(\cY, \cX)$ of $\cX$ relative to $\cY$ (see \cite{RG3} for a discussion of perceptions). In other terms the perception $\RuHom(-,\cX)$ of $\cX$ originates in part from $i_{\cX}\Pi_{\infty,\cX}\widehat{\cX^{op}}$.\\

\newpage

If we apply this in particular to a description of all natural phenomena, one can envision that laws of nature be modeled by simplicial algebras, objects in $\skCAlg$ for $k$ a commutative ring, realized via derived stacks, the collection of which is a Segal topos $\dStk$, to which we can apply all this formalism. One advantage of using stacks to model algebraic realizations of natural laws is that one can completely bypass the use of "target spaces" and "fields" living on such spaces, whose dynamics is given by equations of motion derived from an independent "Lagrangian", notwithstanding the fact that such target spaces for the most part would be obtained from compactification(s). Here dynamics is dictated by the functoriality of stacks themselves, which we interpret as providing coherent manifestations of natural laws. All the derived stacks $F: \skCAlg \rarr \SetD$ together form a Segal topos $\cX = \dStk$. A holistic picture corresponds to considering all stacks and the interactions between them. This is provided by $\RuHom^*_{\SeT}(\cX, \cX)$, another Segal category which gives us a functorial presentation of $\cX$, and containing $\cX$ itself since $id_{\cX} \in \RuHom^*_{\SeT}(\cX, \cX)$. Thus we have two different pictures, a local one given by $\cX = \dStk$, where individual stacks represent single phenomena, and a global one given by $\RuHom^*_{\SeT}(\cX,\cX)$, where all phenomena are considered along with their interactions, which can be regarded as a universal theory. This approach has the advantage of repackaging all of Physics in a purely Algebro-Geometric object such as a Segal Topos, in the spirit of Kontsevich's take on the Mirror Symmetry problem by introducing the Homological Mirror Symmetry formalism (\cite{KS}).\\

Our references for Segal categories are standard: \cite{TV1}, \cite{T}, \cite{T2}, \cite{P}, \cite{HS}, \cite{Si}. For derived stacks we use the foundational papers \cite{TV}, \cite{TV1}, \cite{TV4}, \cite{TV6}, \cite{T}, \cite{T2}.

\begin{acknowledgments}
The author would like to thank J. Bergner and M. Hoyois for useful exchanges, as well as the organizers of the conference ``Exchange of Mathematical Ideas - 2016" at Prescott where this work was completed. The author would like to thank J. Gemmer in particular for stimulating conversations.
\end{acknowledgments}

\newpage

\section{Grothendieck's take on Galois Theory}
We briefly remind the reader of the following fundamental result of Galois Theory: for $K$ a finite Galois extension of a field $k$, $G = \Gal(K/k)$, we have a one-to-one correspondence between subfield extensions $k \subset E \subset K$ and subgroups $H \subset G$, given by $E \mapsto \Gal(K/E)$ and $H \mapsto K^H$. It was Grothendieck's idea in SGA 1 to categorize this result, as clearly recounted in \cite{D}, by first using the well-known fact that there is a one-to-one correspondence between conjugacy classes of subgroups of $G$ and isomorphism classes of transitive $G$-sets, whose category we denote $\tSetsG$, and by regarding field extensions as a category $\cC^{\op}$ whose objects are fields, $k$ then becoming the terminal object of $\cC$. One then moves into the categorical realm. An object $A$ of $\cC$ being fixed, and under the assumption that for all objects $X$ of $\cC$, there is a morphism $A \rarr X$, which is further a strict epi, meaning the joint coequalizer of all the parallel pairs that it coequalizes, if we assume we have a notion of categorical quotient $A \rarr A/H$, preserved by $\Hom_{\cC}(A,-) = [A,-]$, and finally if we assume that $\text{End}(A) = \text{Aut}(A)$, then under those assumptions we have an adjoint equivalence:
\beq
[A,-]: \cC \rightleftarrows \tSetsG: A \times_G - \label{starSeT}
\eeq
thereby providing an abstraction of the classical Galois statement. Grothendieck's idea then was to observe that $[A,-]$ being a map from $\cC$ to $\Set$, one may as well start from a fiber functor $F: \cC \rarr \fSets$, left exact among other things, and under mild conditions on $\cC$ as well as on $F$ itself, one obtains a generalization of the above result (\cite{G}, \cite{D}). One would first show that $F$ is pro-corepresentable, $F = [P, -]$, from which one would make a transition from the result given in \eqref{starSeT} to the following adjoint equivalence:
\beq
F = \Hom_{\cC}(P,-) = [P,-]: \cC \rightleftarrows \fSets^{\pi}:P \times_{\pi} - \nonumber
\eeq
$\pi = \text{Aut}(P)^{\op}$ a profinite group. Grothendieck then observed that this depended of course on $F$ and that a way to make this independent of the choice of a fiber functor was to consider $\Gamma = \{F: \cC \rarr \fSets \}$, which he argued is a groupoid. Considering local systems on $\Gamma$, meaning introducing $i = \Hom_{\cC}(-,\fSets)$, he then proved in a few lines that $i\Gamma \simeq \cC$.

\section{Grothendieck's Galois Theory for Segal Topos}
We now promote $\cC$ and $\fSets$ to the status of Segal topoi as introduced in \cite{TV1}. Thus our base category $\cC$ becomes a fixed Segal topos $\cX$, and we consider functors valued in any Segal topos $T$. Fiber functors in the Segal topos setting would be represented by left exact left adjoints. This justifies the definition:
\beq
\RHomLgXT = \RHomgTX  \nonumber
\eeq
as given in \cite{TV1} where $\RHomg$ is denoted $\RHomgeom$ instead, and \newline $\RuHom^*_{\SeT}(\cX,T)$ refers to the sub-Segal category of $\RuHom(\cX,T)$ spanned by left exact left adjoints. Recall from the same paper that $\SePC$, the category of Segal pre-categories, is a symmetric monoidal model category (\cite{Ho}) with the direct product as monoidal product, hence $\Ho(\SePC)$ has an internal Hom object (\cite{Ho}) denoted $\RHom(A,B) \in \Ho(\SePC)$ for $A,B \in \Ho(\SePC)$, where $\RHom(A,B) \cong \uHom(A,RB)$, $\uHom$ the internal Hom object in $\SePC$. However, we also have an anti-equivalence $\RHomLgXT \simeq \RHomRgTX^{\op}$. That this is an anti-equivalence follows from the commutative diagram below; for $\cX$ and $\cY$ two Segal topos, $f,g: \cX \rarr \cY$ two geometric morphisms, $\alpha: f^* \Rarr g^*$ a morphism in $\RuHom^*_{\SeT}(\cY, \cX)$, $x \in \cX_0$ and $y \in \cY_0$, there corresponds $\beta: g_* \Rarr f_*$, a morphism in $\RuHom_{*, \SeT}(\cX, \cY)$ as in:
\beq
\begin{CD}
	\cX_{(g^* y, x)} @>\cong>> \cY_{(y, g_*x)} \\
@V\alpha^*_yVV @VV\beta_{x,*}V \\
	\cX_{(f^* y, x)} @>>\cong> \cY_{(y, f_*x)}
\end{CD} \nonumber
\eeq
Hence we can also define $\RHomgTX$ as:
\beq
\RHomgTX = \RHomRgTX^{\op} \nonumber
\eeq
\begin{PiGpd}
For $\cX$ and $T$ two Segal topos, $\RHomgTX$ is a Segal groupoid.
\end{PiGpd}
\begin{proof}
	$\cX$ being a Segal topos, by definition there is a fully faithful map $i: \cX \rarr \hat{B} = \RHom(B^{\op}, \Top)$, $B$ a Segal Category, $\Top = \Loc(*) = L(\SetD)$. This map induces a fully faithful morphism of Segal categories $i_*: \RuHom(T, \cX) \rarr \RuHom(T, \hat{B})$. As a matter of fact, since right adjoints with a left exact left adjoint map to right adjoints with a left exact left adjoint, $i_*: \RuHom_{*, \SeT}(T, \cX) \rarr \RuHom_{*, \SeT}(T, \hat{B})$ is also fully faithful (this is the object of the lemmas that follow). Now using the adjunction formula $\RHom(A \times B, C) \simeq \RHom(A, \RHom(B,C))$ of \cite{TV1} limited to right adjoints with a left exact left adjoint:
\begin{align}
\RHomRg(T, \hat{B}) &=\RHomRg(T, \RHom(B^{\op}, \Top)) \nonumber \\
& \cong \RHomRg(T \times ^{\bL} B^{\op}, \Top) \nonumber \\
& = \RHom( B^{\op}, \RHomRg(T, \Top)) \nonumber \\
&= \RHom (B^{\op}, \RHomg(T, \Top)^{\op}) \nonumber
\end{align}
where on the second line we just have a left exact left adjoint between $T$ and $\Top$. We then use the fact that $\RHomg(T, \Top)$ is a Segal groupoid as proved in \cite{TV1}, so that $\RHom(B^{\op}, \RHomg(T, \Top)^{\op})$ itself is a Segal groupoid, hence so is $\RHomRg(T, \hat{B})$, and $\RHomRgTX$ faithfully maps into it, so is a Segal groupoid, or equivalently, $\RHomgTX$ is a Segal groupoid.
\end{proof}

In the proof we used the fact that for $A,B \in \Ho(\SeCat)$, $\RuHom(A,B) \in \Ho(\SeCat)$. From \cite{T}, $B \in \SeCat$ implies we have a fully faithful embedding $B \rarr LM$ for some model category $M$. It follows we have a fully faithful embedding $\RuHom(A,B) \rarr \RuHom(A,LM) \cong L(M^A)$, where we used the strictification theorem. Note this holds for $A \in \CatD$, the category of simplicial categories. However $\Ho(\SeCat) \simeq \Ho(\CatD)$, so we can regard $A$ as a simplicial category. Further, $L(M^A) \in \Ho(\CatD) \simeq \Ho(\SeCat)$, hence we regard $\RuHom(A,B)$ as a Segal category.\\

We can also mention in passing that for $A,B \in \Ho(\SeCat)$, a morphism $f: A \rarr B$ is regarded as an object of $\RuHom(A,B)$. Observe as in \cite{Si} that if $[0]$ is the category with one object 0, regarded as a Segal category, then by adjunction $\Hom_{\Ho(\SeCat)}(A \times [0] , B) \cong \Hom_{\Ho(\SeCat)}([0], \RuHom(A,B))$. An object in the left hand side Hom corresponds to a map of bisimplicial sets, while an object in the right hand Hom corresponds to picking an object of $\RuHom(A,B)$. In the same manner, if $[1]$ denotes the category with two objects 0 and 1 and a unique morphism $0 \rarr 1$, then for $f,g:A \rarr B$, a natural transformation $\alpha: f \Rarr g$ can be seen as an object of $\Hom_{\Ho(\SeCat)}(A \times [1], B)$, or by adjunction as an object of $\Hom_{\Ho(\SeCat)}([1], \RuHom(A,B))$, that is an object of $\RuHom(A,B)_{(f,g)}$.\\ 

We now go over this claim used in the proof of the above theorem that $\RuHom_{*,\SeT}(T,\cX) \rarr \RuHom_{*,\SeT}(T,\hat{B})$ is fully faithful. First, we have to understand what is a morphism of adjunctions. We follow \cite{McL}. Given two adjunctions $F: X \adj A: G$ and $F': X \adj A: G'$, $\sigma: F \Rarr F'$ and $\tau: G' \Rarr G$ are said to be conjugate when the following diagram commutes for all $x \in X$ and $a \in A$:
\beq
\xymatrix{
	A(F'x,a) \ar[d]_{(\sigma_x)^*} \ar[r]^{\cong} & X(x,G'a) \ar[d]^{(\tau_a)_*} \\
	A(Fx,a) \ar[r]_{\cong} &X(x,Ga)
}\nonumber
\eeq
We regard a conjugate pair $(\sigma,\tau)$ of natural transformations as a map of adjunctions from $(F,G)$ to $(F',G')$. Now recall that we want to show \newline $\RuHom_{*,\SeT}(T,\cX) \xrarr{i_*} \RuHom_{*,\SeT}(T,\hat{B})$ is fully faithful. We consider $F,G$ in $\RuHom_{*,\SeT}(T,\cX)$ and $\varphi: F \Rarr G$ in $\RuHom_{*,\SeT}(T,\cX)$. By the previous discussion, $\varphi$ should be part of a conjugate pair of natural transformations. To be precise, we consider a map of adjunctions $(G',G) \rarr (F',F)$, with $F': \cX \adj T:F$ and $G': \cX \adj T: G$, with $F'$ and $G'$ left exact, $\varphi': G' \Rarr F'$ and $\varphi: F \Rarr G$. Diagrammatically, this means the following diagram is commutative:
\beq
\xymatrix{
	T_{(F'x,t)} \ar[d]_{(\varphi'_x)^*} \ar[r]^{\cong} & \cX_{(x,Ft)} \ar[d]^{(\varphi_t)_*} \\
	T_{(G'x,t)} \ar[r]_{\cong} &\cX_{(x,Gt)}
}\nonumber
\eeq
which, using compositions in Segal categories should read like so:
\beq
\xymatrix{
	T_{(G'x,F'x)} \times T_{(F'x,t)} \ar[d] \ar[r]^{\cong} & \cX_{(x,Ft)} \times \cX_{(Ft,Gt)} \ar[d] \\
	T_{(G'x,t)} \ar[r]_{\cong} & \cX_{(x,Gt)}
}\nonumber
\eeq
Knowing that both $F'$ and $G'$ are left exact, and $\varphi': G' \Rarr F'$ (that is for all $x$, $\varphi'_x: G'x \rarr F'x$), for a diagram $x^{\alpha}$ in $\cX$ used to prove left exactness:
\begin{align}
	\varphi'_{\lim x^{\alpha}} & = G'(\lim x^{\alpha}) \rarr F'(\lim x^{\alpha}) \nonumber \\
	&= \lim G' x^{\alpha} \rarr \lim F' x^{\alpha} \nonumber \\
	& = \lim(G' x^{\alpha} \rarr F' x^{\alpha}) \nonumber \\
	&= \lim \varphi'_{x^{\alpha}} \nonumber
\end{align}
thus $\varphi'$ is left exact as well. To conclude, if $\varphi: F \Rarr G$ and $\varphi': G' \Rarr F'$ form a conjugate pair in $\RuHom(T,\cX)$, then $\varphi'$ is left exact.\\
\begin{istarwelldef}
	In the setting of the above theorem, $i_*: \RuHom_{*,\SeT}(T,\cX) \rarr \RuHom_{*,\SeT}(T,\hat{B})$ is well-defined.
\end{istarwelldef}
\begin{proof}
	Let $F \in \RuHom_{*,\SeT}(T,\cX)$. We first show $i_*F$ is a right adjoint with a left exact left adjoint. Since $\cX \in \SeT$, there is a fully faithful map $i: \cX \rarr \hat{B}$ as we saw in the proof of the above theorem, but we also have a left exact left adjoint $j: \hat{B} \rarr \cX$. $F$ itself has a left exact left adjoint that we denote by $F'$. We consider the following compositions:
	\beq
		T \adj^F_{F'} \cX \adj^i_j \hat{B} \nonumber
	\eeq
We argue $j^*F' \dashv i_*F$ in $\RuHom(T,\hat{B})$, and $j^*F'$ is left exact. We first show the adjointness. It suffices to write:
	\begin{align}
		T_{(F'\circ jb,t)} &\cong \cX_{(jb,Ft)} \nonumber \\ 
		&\cong \hat{B}_{(b,iFt)} \nonumber
	\end{align}
	We now show $j^*F'$ is left exact. Consider a diagram in $\hat{B}$. We have $j^*F'(\lim b^{\alpha}) = F'\circ j(\lim b^{\alpha}) = F'(\lim jb^{\alpha})$ since $j$ is left exact, and this is further equal to $\lim(F'\circ j b^{\alpha})$ since $F'$ is left exact as well. But this is $\lim j^*F'(b^{\alpha})$. Thus $i_*$ maps objects of $\RuHom_{*,\SeT}(T,\cX)$ to objects of $\RuHom_{*,\SeT}(T,\hat{B})$.\\

We now show $i_*$ maps morphisms in $\RuHom_{*,\SeT}(T,\cX)$ to morphisms of $\RuHom_{*,\SeT}(T,\hat{B})$. Let $F,G$ in $\RuHom_{*,\SeT}(T,\cX)$ with respective left exact left adjoints $F',G'$ in $\RuHom^*_{\SeT}(\cX,T)$, let $\varphi: F \Rarr G$ and $\varphi': G' \Rarr F'$ (left exact) be a conjugate pair for those adjoints. We need to show that $\phi = i_* \varphi: i_* F \Rarr i_*G$ has a left exact conjugate in $\RuHom^*_{\SeT}(\hat{B},T)$. We have the following compositions:
	\beq
	\xymatrix{
		T \ar@/^2pc/[rr]^{i \circ F} \ar@/^/[r]^F & \cX \ar@/^/[l]^{F'} \ar@/^/[r]^i & \hat{B} \ar@/^2pc/[ll]^{F'\circ j} \ar@/^/[l]^j
		} \nonumber
	\eeq
We have $i_* \varphi: i_*F \Rarr i_*G$ and $j^* \varphi': j^*G' \Rarr j^*F'$. We claim $(j^*\varphi', i_* \varphi)$ is a conjugate pair with $j^* \varphi'$ left exact. We want the following diagram to commute:
\beq
	\xymatrix{
		T_{(j^*F'b,t)} \ar[d]_{(j^*\varphi'_b)^*} \ar[r]^{\cong} & \hat{B}_{(b,i_*Ft)} \ar[d]^{(i_*\varphi_t)_*} \\
		T_{(j^*G'b,t)} \ar[r]_{\cong} & \hat{B}_{(b,i_*Gt)}
		}\nonumber
\eeq
We have:
\beq
	\xymatrix{
	T_{(j^*F'b,t)} \ar[d]_{(j^*\varphi'_b)^*} \ar@{=}[r] & T_{(F'jb,t)} \ar[d]_{(\varphi'_{jb})^*} \ar[r]^{\cong} &\cX_{(jb,Ft)} \ar[d]^{(\varphi_t)_*} \ar[r]^{\cong} &\hat{B}_{(b,iFt)} \ar[d]^{(i_* \varphi_t)_*} \\
	T_{(j^*G'b,t)} \ar@{=}[r] & T_{(G'jb,t)} \ar[r]_{\cong} &\cX_{(jb,Gt)} \ar[r]_{\cong} &\hat{B}_{(b,iGt)}
	}\nonumber
	\eeq
	The two squares on the right commute by adjunction, which gives us our desired commutativity, i.e. $j^* \varphi'$ and $i_* \varphi$ form a conjugate pair. We also have that $j^* \varphi'$ is left exact as a composite of left exact functors. To summarize, for $\varphi$ a morphism in $\RuHom_{*,\SeT}(T,\cX)$, part of a conjugate pair $(\varphi', \varphi)$ with $\varphi'$ left exact, $i_* \varphi$ is part of a conjugate pair $(j^*\varphi', i_* \varphi)$ in $\RuHom(T,\hat{B})$, that is morphisms in $\RuHom_{*,\SeT}(T,\cX)$ map to morphisms in $\RuHom_{*,\SeT}(T,\hat{B})$.
\end{proof}
\begin{istarff}
	In the setting of the above theorem, $i_*: \RuHom_{*,\SeT}(T,\cX) \rarr \RuHom_{*,\SeT}(T, \hat{B})$ is fully faithful.
\end{istarff}
\begin{proof}
	Let $F,G \in \RuHom_{*,\SeT}(T,\cX)$, $\phi: i_*F \rarr i_*G$ in $\RuHom_{*,\SeT}(T, \hat{B})$. We show we have a morphism $\varphi: F \Rarr G$ in $\RuHom_{*,\SeT}(T,\cX)$. Note that $\phi$ is a morphism in $\RuHom(T,\hat{B})$, so $i_*: \RuHom(T,\cX) \rarr \RuHom(T,\hat{B})$ being fully faithful, there is some $\varphi: F \Rarr G$, morphism in $\RuHom(T,\cX)$. We show it has a left exact conjugate $\varphi'$. We use the fact that $F$ and $G$ have left exact left adjoints that we denote by $F'$ and $G'$ respectively. Since $j \dashv i$, we have from earlier work that $j^*F' \dashv i_*F$ and $j^*G' \dashv i_*G$ in $\RuHom(T,\hat{B})$. Since $\phi$ is a morphism in $\RuHom_{*,\SeT}(T,\hat{B})$ from $i_*F$ to $i_*G$, it has a left exact conjugate $\phi': j^*G' \Rarr j^*F'$. At the level of adjoints, we have $T_{(j^*F'b,t)} \cong \hat{B}_{(b,i_*Ft)}$, which we can write as $T_{(F'jb,t)} \cong \hat{B}_{(b,iFt)}$. Letting $b = ix$, this reads:
	\beq
	T_{(F'jix,t)} \cong \hat{B}_{(ix,iFt)} \cong \cX_{(x,Ft)} \nonumber
	\eeq
	where we used the fact that $i$ is fully faithful. It follows $T_{(F'jix,t)} \cong \cX_{(x,Ft)}$, hence $F'ji \simeq F'$ up to equivalence. Since we work in $\Ho(\SeCat)$, we can therefore identify $F'ji$ with $F'$, and $G'ji$ with $G'$. We have the following commutative diagram:
	\beq
	\xymatrix{
		T_{(F'jix,t)} \ar[d]_{(\phi'_{ix})^*} \ar[r]^{\cong} & \cX_{(x,Ft)} \ar[d]^{(\varphi_t)_*} \\
		T_{(G'jix,t)} \ar[r]_{\cong} & \cX_{(x,Gt)}
		}\nonumber
	\eeq
that is $\phi'_{ix}: G'x \rarr F'x$, or $i_*\phi' = \varphi': G' \Rarr F'$ is conjugate to $\varphi$, and also left exact by composition, since $i$ right adjoint preserves limits, and $\phi'$ is left exact. Thus $i_*$ is full. For faithfulness, let $\varphi_1, \varphi_2: F \Rarr G$ in $\RuHom_{*,\SeT}(T, \cX)$. Suppose $i_* \varphi_1 = i_* \varphi_2 = \phi: i_*F \Rarr i_*G$. Pointwise this means $\phi_t: iFt \rarr iGt$, and $i$ being faithful, we have a unique map $\varphi_t: Ft \rarr Gt$, and this for any $t$, hence $\varphi_1 = \varphi_2$.  
\end{proof}

Regarding notations, since $\hatXop = \RHom(\cX, T)$ relative to $T$, and given that $\RHomLgXT$ is a full-sub Segal category thereof (\cite{TV1}), we will write $\PiTX$ for $\RuHom^*_{\SeT}(\cX, T) = \RHomgTX$, since we just proved that it is a Segal groupoid. This then defines a functor, for any Segal topos $T$:
\begin{align}
	\PiThat: \Ho(\SeT)^{\op} &\rarr \Ho(\SeGpd) \nonumber \\
\cX &\mapsto \PiTX \nonumber
\end{align}
where $\SeGpd$ denotes the category of Segal groupoids. We now define:
\begin{align}
	i_T: \Ho(\SeGpd)^{\op} &\rarr \Ho(\SeCat) \nonumber \\
A &\mapsto i_T(A) = \RHom(A,T) \nonumber
\end{align}
We claim $i_T$ is actually valued in $\Ho(\SeT)$ if $T$ is a Segal topos:
\begin{RuHomATSeT}
	If $T \in \SeT$, $A \in \SeCat$, then $i_T A = \RuHom(A,T) \in \Ho(\SeT)$.
\end{RuHomATSeT}
\begin{proof}
	Since $T \in \SeT$, we know there is a fully faithful embedding $i: T \rarr \hat{B} = \RuHom(B^{\op}, \Top)$, with a left exact left adjoint $j$. The map $i$ induces a fully faithful map:
	\begin{align}
		i_*: \RuHom(A,T) \rarr & \, \RuHom(A, \hat{B}) \nonumber \\
		&=\RuHom(A, \RuHom(B^{\op}, \Top)) \nonumber \\
		&\cong \RuHom(A \times^{\bL} B^{\op}, \Top) \nonumber 
	\end{align}

\newpage

We now determine whether $i_*$ has a left exact left adjoint. We show it is $j_*: i_{\hat{B}} A \rarr i_T A$. We want $\Hom_{i_T A}(j_*F, G) \cong \Hom_{i_{\hat{B}}A}(F, i_*G)$. Consider:
	\begin{align}
		\Hom_{i_T A}(j_*F, G) & = \Hom_{i_T A}(j \circ F, G) \nonumber \\
		&=\coprod_{a \in A} \Hom_T((j \circ F)(a), Ga) \nonumber \\
		&=\coprod_a \Hom_T(j(Fa), Ga) \nonumber \\
		&\cong \coprod_a \Hom_{\hat{B}}(Fa, i(Ga)) \nonumber \\
		&=\Hom_{i_{\hat{B}}A}(F,i_*G) \nonumber
	\end{align}
where we have used the fact that natural transformations between functors are defined pointwise. Thus $j_* \dashv i_*$. We now argue $j_*$ is left exact. A diagram of objects in $i_{\hat{B}} A$ is defined pointwise over $A$ hence:
	\begin{align}
		j_* \lim F^{\alpha} &=j \circ \lim F^{\alpha} \nonumber \\
		&= \coprod_{a \in A} j \circ \lim F^{\alpha}(a) \nonumber \\
		&= \coprod_a \lim j \circ F^{\alpha}(a) \nonumber \\
		&= \lim j_*F^{\alpha} \nonumber
	\end{align}
where we used the fact that $j$ is left exact. Thus $j_*$ is left exact, which completes the proof.
\end{proof}
\begin{iTLadj}
For $T \in \SeT$, we have an adjunction:
\beq
	i_T: \Ho(\SeGpd)^{\op} \rightleftarrows \Ho(\SeT): (\PiThat)^{\op} \nonumber
\eeq
\beq
	(\PiThat)^{\op} \dashv i_T \nonumber
\eeq
It follows that for all Segal topos $\cX$ we have a unique unit map up to homotopy:
\beq
	\cX \rarr i_T \Pi_{\infty,T}\widehat{\cX^{\op}}  \nonumber
\eeq
\end{iTLadj}
\begin{proof}
	It suffices to write, for $A,T  \in \Ho(\SeT)$, $B \in \Ho(\SeGpd)$:
\begin{align}
	\Hom^*_{\Ho(\SeT), lex}(A, i_T B) &= \Hom^*_{\Ho(\SeT), lex}(A, \RuHom(B,T)) \nonumber \\
	&\cong \Hom^*_{\Ho(\SeT), lex}(A \times^{\bL} B, T) \nonumber \\
	&\cong \Hom_{\Ho(\SeGpd)}(B, \RuHom^*_{\SeT}(A,T)) \nonumber \\
	&= \Hom_{\Ho(\SeGpd)}(B, \Pi_{\infty,T}\widehat{A^{\op}}) \nonumber \\
	&=\Hom_{\Ho(\SeGpd)^{\op}}(\Pi_{\infty,T}\widehat{A^{\op}},B) \nonumber
\end{align}
where on the second line the left exactness is between $A$ and $T$.
\end{proof}
It follows that for $T \in \SeT$, $\cX \rarr i_T$ is a universal arrow. In particular for $T = \cX$, we have a universal arrow from $\cX$ to its homotopical, internal perception $i_{\cX} = \RuHom(-,\cX)$. Recall that we regard a perception $\RuHom(-,X)$ as providing a representation of the object $\cX$. That the morphism $\cX \rarr \RuHom(-,\cX)$ is a universal arrow supports this concept. Now by universality, for $\cY \in \Ho(\SeGpd)$, we have a unique map $\Pi_{\infty,\cX} \widehat{\cX^{\op}} \rarr \cY$ in $\Ho(\SeGpd)^{\op}$ making the following diagram commutative:
\beq
\xymatrix{
	\cX \ar[r] \ar[dr] & \RuHom(\Pi_{\infty,\cX} \widehat{\cX^{\op}}, \cX) \ar@{.>}[d] \\
	& \RuHom(\cY, \cX)
} \nonumber
\eeq
In such a diagram, $\RuHom(\cY,\cX)$ is the internal perception of $\cX$ relative to $\cY \in \Ho(\SeGpd)$. The morphism $\cX \rarr \RuHom(\cY, \cX)$ is a transition morphism from $\cX$ to its perception relative to $\cY$. That the above diagram is commutative says that this part of $\RuHom(\cY,\cX)$ that originates from $\cX$ can be obtained from $i_{\cX}\Pi_{\infty,\cX}\widehat{\cX^{\op}}$.\\ 

\newpage

\section{Natural Phenomena}
We regard natural phenomena as following some basic rules, or fundamental laws, encoded in algebras. Coherent manifestations of such laws will be modeled by stacks valued in simplicial sets, and according to the philosophy of derived algebraic geometry, it is natural then to consider simplicial algebras. Hence we start with a commutative ring $k$, we let $\kMod$ be the category of $k$-modules, commutative monoids of which form $\skCAlg$ the category of simplicial $k$-algebras. Its opposite category $\dkAff = L(\skCAlg)^{\op}$ is referred to as the Segal category of derived affine stacks (\cite{TV2}, \cite{TV}, \cite{TV4}, \cite{T}). We put the ffqc topology on this Segal category. The category $\cX = \dStk = \dkAff^{\sim, ffqc}$ of affine stacks is a localization of the Segal category of pre-stacks $\widehat{\dkAff} = \RHom(\dkAff^{\op}, \Top)$. The former is a Segal topos (\cite{T}), hence all the above formalism applies to $\dStk$. In particular by universality of the unit map $\cX \rarr i_{\cX}\Pi_{\infty,\cX}\widehat{\cX^{\op}}$, we have a morphism $i_{\cX}\Pi_{\infty,\cX}\widehat{\cX^{\op}} \rarr \RuHom(\cY, \cX)$ in $\Ho(\SeT)$ to the homotopy internal perception of $\cX$ relative to $\cY$, for any $\cY \in \Ho(\SeGpd)$.\\

$\cX$ being a Segal category, we know there is a Segal category $\cU$ such that $\cU = \RuHom^*_{\SeT}(\cX,\cX) = \Pi_{\infty,\cX}\widehat{\cX^{op}}$. In contrast to just having the Segal category $\cX$, with its objects and morphisms, $\cU$ provides a functorial presentation of $\cX$ in the following sense. Consider $\psi: \cX \rarr \cX$ a geometric morphism in $\cU$, let $F,G \in \cX$. Then we have an induced morphism of simplicial sets $\psi: \cX_{(F,G)} \rarr \cX_{(\psi F, \psi G)}$. Thus $\cU$ presents $\cX$ as a cohesive whole, whereby all morphisms, that is all dynamics within $\cX$, are related whenever a connection exists. Further, $\psi = id_{\cX}$ is clearly an element of $\cU$. However $id_{\cX}$ just reproduces $\cX$. Thus we can morally see $\cX$ as an object of $\cU$, or $\cU$ as an enlargement of $\cX$. It follows that a comprehensive view of $\cX$ occurs in a higher Segal category $\cU$. Observe that this can be repeated: $\RuHom^*_{\SeT}(\cU, \cU)$ is yet another Segal category, in which $\cU$ would take its full meaning.\\ 

\newpage

Recall that one representation of $\cX$ is provided by its homotopical internal perception $\RuHom(-,\cX)$. By the universality of $\cX \rarr i_{\cX} \Pi_{\infty,\cX} \widehat{\cX^{op}}$, for any morphism $f: \cX \rarr \RuHom(\cY, \cX)$ for $\cY \in \Ho(\SeGpd)$, which corresponds to focusing on one facet of $\cX$ relative to $\cY$, we have a unique morphism $ \phi: \cY \rarr \Pi_{\infty,\cX}\widehat{\cX^{\op}}=\cU$ such that the following diagram is commutative:
\beq
\xymatrix{
	\cX \ar[r] \ar[dr] & \RuHom(\cU,\cX) \ar@{.>}[d] \\
	& \RuHom(\cY,\cX)
} \nonumber
\eeq
$\phi$ provides us with that information of $\cU$ which is originating from $\cY$. The commutative diagram tells us the perception of $\cX$ relative to $\cY$ factors through the perception of $\cX$ relative to $\cU$. In other terms, any perception of $\cX$ by Segal groupoids factors through $i_{\cX}\Pi_{\infty, \cX}\widehat{\cX^{\op}}$.\\

One can contrast this with the concept of having a field in Physics being defined on a target space, itself resulting from a compactification. Here the field is replaced by a functor on simplicial algebras and with target space $\SetD$. Those functors provide a manifestation of natural laws. The collection of all such functors is the Segal Topos $\dStk$, which along with $\RuHom^*_{\SeT}(\dStk, \dStk)$ would consequently provide a universal theory of physical phenomena.

\end{document}